\documentclass[12pt,a4paper]{article}
\usepackage[latin1]{inputenc}
\usepackage{amsmath}
\usepackage{amsfonts}
\usepackage{amssymb}
\usepackage{amsthm}
\usepackage{verbatim}
\numberwithin{equation}{section}
\newtheorem{lm}{Lemma}
\newtheorem{thrm}{Theorem}
\newtheorem{cor}{Corollary}
\newtheorem{defn}{Definition}

\newcommand{\domain}{\Omega}
\newcommand{\fb}{F\!B}
\newcommand{\ep}{\varepsilon}

\begin{document}
\author{Thomas Backing\thanks{Partially supported by NSF grant DMS-1101246}\\ Purdue University}
\date{}

\title{Regularity of Solutions to a Parabolic Free Boundary Problem with Variable Coefficients}

\maketitle

\section{Introduction}
Free boundary problems occur in many science and engineering problems when a law changes discontinuously as it crosses from one region to another. Typical problems require the solution to satisfy different sets of conditions on its positivity and negativity sets. Since these regions depend on the solution $u$ itself, the interface between them, called the free boundary, is not known at the outset.

In this paper, we study the regularity of viscosity solutions to a family of parabolic free boundary problems of the form
\begin{equation}\label{FBP_Statement}
\left\lbrace
\begin{aligned}
&\mathcal{L}u -u_t =0 \quad \text{in } \;(\{u>0\}\cup\{u\leq 0\}^\circ)\subset \domain \\
 &G(u^+_\nu,u^-_\nu)=1 \quad \text{along } \;\partial\{u>0\} \subset \domain.
\end{aligned}
\right.
\
\end{equation}

Here $\mathcal{L}$ is a non-divergence form uniformly elliptic operator and $u^\pm_\nu$ denote, respectively, the inner normal derivative relative to the sets $\{u>0\}$ and $\{u\leq 0\}^\circ$.

The main goal of this paper is to prove that, under suitable conditions on $G$ (stated precisely in the next section), viscosity solutions to~\eqref{FBP_Statement} which possess a Lipschitz free boundary and satisfy a non-degeneracy condition are Lipschitz continuous. This is the optimal regularity for this problem.

Motivating examples of the free boundary condition $G(u^+_\nu,u^-_\nu) $ include $u^+_\nu =1$  or $(u^+_\nu)^2-(u^-_\nu)^2 =2M>0$, both of which arise from a singular perturbation problem used to model a problem from combustion theory. This problem consists of studying the limit as $\ep \rightarrow 0$ of weak solutions to
\[
\Delta u^\ep -u^\ep_t = \beta_\ep(u^\ep)
\]
where $\beta_\ep(s) =\frac{1}{\ep}\beta(s/\ep)$ for a Lipschitz function $\beta$ supported on $[0,1]$ with $\beta_\ep$ satisfying for some positive constants $C$ and $M$
\[
0\leq \beta_\ep \leq \frac{C}{\ep}\chi_{(0,\ep)} \quad \text{and} \quad \int_0^\ep \beta_\ep(s)\, \mathrm{d}s =M.
\]
The one-phase version of this problem (i.e. $u^\ep \geq 0$) was studied in \cite{CV}.  It was shown there that the limit function $u$ satisfies $u^+_\nu=1$ along the boundary of its positivity set. For two phase problem (studied in \cite{CLW1}, \cite{CLW2}) the boundary condition for the limit $u$ is $(u^+_\nu)^2-(u^-_\nu)^2 =2M>0$. In all of these works the boundary condition was shown to hold in a weak sense, with pointwise equality only at regular points on the zero set. This two phase boundary condition is the prototype for the function $G$.

An elliptic version of~\eqref{FBP_Statement} with the Laplacian was studied in \cite{C1}, \cite{C2}. It is in these works that the main ideas used in this paper, such as monotonicity cones, viscosity solutions to~\eqref{FBP_Statement}, and `sup-convolutions' were first developed. Similar methods were later applied to the study of the Stefan problem for the heat equation in \cite{ACS1}, \cite{ACS2}, \cite{ACS3}.  All of these works, including the singular perturbation problem, involved only the case where $\mathcal{L} =\Delta$, $\Delta$ denoting the Laplacian. In [F] these results were adapted to the study of the constant coefficient version of~\eqref{FBP_Statement}.

When $\mathcal{L}=\Delta$, whether in the elliptic or parabolic case, extensive use is made of the fact that directional derivatives of solutions to a constant coefficient linear PDE are themselves solutions to the same PDE. In particular, tools like the Harnack Inequality can be applied to the directional derivatives. At variance with the case when $\mathcal{L}=\Delta$, directional derivatives of solutions to the operator $\mathcal{L}-\partial_t$ are not themselves solutions. This prevents a straightforward extension of the constant coefficient results to this variable coefficient case.
Indeed, the only results extending these methods to a variable coefficient parabolic problem are the recent papers \cite{FS1}, \cite{FS2} for the Stefan problem.

The strategy of this work is to begin by deducing that a cone of $\ep$-monotonicity (defined below) exists for solutions of~\eqref{FBP_Statement}. This is due to $u$ being $\mathcal{L}$-caloric on a Lipschitz domain and vanishing on a portion of its boundary. We then construct an iteration which simultaneously decreases both $\ep$ and the aperture of the $\ep$-monotonicity cone opening. By carefully balancing decrease in $\ep$, which is desirable, with the decrease in the cone opening, which is not, we show that there does indeed exist a cone of full monotonicity  up to the free boundary of $u$. Using this information, we deduce control of the time derivative by the spacial gradient $|\nabla u|$, then proceed to show the boundedness of $|\nabla u|$, thereby establishing Lipschitz continuity of the solution.

The structure of this work is as follows: Section 2 states precisely the problem under consideration as well as our main result. Section 3 collects the main tools and known results used the analysis of this problem. Section 4 contains our results on the asymptotic behavior of solutions near the free boundary. These results are used in Section 5 to prove that a space-time cone of monotonicity exists up to the free boundary for $u$. In Section 6, we use this to prove the Lipschitz regularity of $u$.

\section{Definitions and Statement of Results}

We will denote the positivity set of $u$ by $\domain^+$, i.e. $\domain^+=\{x\in\Omega\ |u(x)>0)\}$; likewise the negative set is denoted by $\domain^-$. Occasionally we will write $\domain^\pm(u)$ to emphasize the dependence of these domains on the function $u$. The set $\partial\{u>0\}$ is the free boundary and will be denoted by $\fb(u)$ or just $\fb$.

We will denote by $C_{R,T}(x_0,t_0)$ the cylinder
\[
B'_R(x_0)\times (t_0-T,t_0+T).
\]
If the center of the cylinder is the origin we will simply write $C_{R,T}$ and if $R=T$ we will write $C_R$.

 The operator
\[
\mathcal{L} =\sum_{i,j} a_{ij}(x,t)D_{ij}
\] has Holder continuous coefficients $a_{ij} \in C^{0,\alpha}(\domain)$, $0<\alpha\leq 1$ and there exists $\lambda, \Lambda >0$ such that
\[
\lambda|\xi|^2 \leq \sum a_{ij}(x,t)\xi_i \xi_j \leq \Lambda |\xi|^2
\] for all $(x,t) \in \domain$. Denoting by $A(x,t)$ the matrix $[a_{ij}(x,t)]$,  we assume $A(0,0) =[\delta_{ij}]$, i.e. the identity.

On $G(a,b)$ we will require:
\begin{enumerate}
\item $G$  Lipschitz with constant $L_G$ in both variables.
\item $G(a_1,b) -G(a_2,b) >c^*(a_1-a_2)^p$ if $a_1>a_2$ (strictly increasing in the first variable)
\item $G(a,b_1)-G(a,b_2) < -c^*(b_1-b_2)^p$ if $b_1>b_2$ (strictly decreasing in the second variable)
\end{enumerate}
The $p$ appearing here is some positive power.
\begin{defn}(Classical Subsolution/Supersolution)
We say $v(x,t)$ is a classical subsolution (supersolution) to~\eqref{FBP_Statement} if $v \in C^1(\overline{\domain^+(v)})\cup C^1(\overline{\domain^-(v)})$, $\mathcal{L}v-v_t\geq 0$ $(\mathcal{L}v-v_t\leq 0)$ in $\domain^\pm(v)$ and
\[
G(v_\nu^+,v_\nu^-) \geq 1 \;\;(G(v_\nu^+,v_\nu^-) \leq 1), \quad \text{where }\nu= \frac{\nabla v^+}{|\nabla v^+|}\]
A strict subsolution (supersolution) satisfies the above with strict inequalities.
\end{defn}

\begin{defn}(Viscosity Subolutions/Supersolutions)
A continuous function $v(x,t)$ is a viscosity subsolution (supersolution) to~\eqref{FBP_Statement} in $\domain$ if for every space-time cylinder $Q= B_r' \times (-T,T) \Subset \domain$ and for every classical supersolution (subsolution) $w$ in $Q$, the inequality $v\leq w$ ($v\geq w)$ on $\partial_pQ$ implies that $v\leq w$ ($v\geq w)$. Additionally, if $w$ is a strict classical supersolution (subsolution), then $v<w$ $(v>w)$ on $\partial_p Q$ implies $v<w$ $(v>w)$ inside $Q$.
\end{defn}

We now turn to the hypotheses on the free boundary of $u$. Our main result will require that this free boundary is Lipschitz, but will we will also require a non-degeneracy condition to hold at regular points. We first define such points.

\begin{defn}(Regular Points) A point $(x_0,t_0)$ on the free boundary of $u$ is a right regular point if there exists a space-time ball $B_R \subset \domain^+$ such that $B_R \cap \partial \{u>0\} =\{(x_0,t_0)\}$.

 A point $(x_0,t_0)$ on the free boundary of $u$ is a left regular point if there exists a space-time ball $B_R \subset \domain^-$ such that $B_R \cap \partial \{u\leq 0\} =\{(x_0,t_0)\}$.

\end{defn}

We now state precisely our assumptions:\\
\\
(H1)\emph{ The free boundary  $FB(u)$ is the graph of a Lipschitz function $f$, that is, $FB(u)= \{(x',x_n,t)| f(x',t)=x_n\}$ with $f(0,0) =0$. We will denote by $L$ and $L_0$ the Lipschitz constant of $f$ in space and time respectively.}\\
(H2) \emph{$u$ satisfies the following non-degeneracy condition: There exists a $m>0$ such that if $(x_0,t_0)$ is a right regular point for $u$ then}
\begin{equation}
\frac{1}{|B'_r(x_0)|}\int_{B'_r(x_0)} u^+ dx \geq mr.
\end{equation}

The main result of this paper is the following theorem.
\begin{thrm} Let $u$ be a viscosity solution to~\eqref{FBP_Statement} in $C_1$ satisfying (H1) and (H2). Then in $C_{1/2}$, $u$ possesses a cone of monotonicity in both space and time and is Lipschitz continuous.
\end{thrm}

 \section{Main Tools}
 In this section we collect some of the essential tools and known results used in the analysis of~\eqref{FBP_Statement}.

Let
\[\Omega_{2r} =\{(x',x_n,t) : |x'|<2L^{-1}r, |t|<4L_0^{-2}r^2, f(x',t)<x_n<4r\}.
\]

Denote by $P_r =(0,r,0)$, $\overline{P_r}=(0,r,2L_0^{-2}r^2)$, $\underline{P_r}=(0,r,-2L_0^{-2}r^2)$. These are the inward point, forward point and backward point, respectively. 

Denote by $\delta(X,Y)$ the parabolic distance between $X=(x,t)$ and $Y=(y,s)$, that is, $\delta(X,Y) = |x-y|+|t-s|^{1/2}$ and by $\delta_X$ the parabolic distance from $X$ to the origin.\\

Our tools, valid for $\mathcal{L}$-caloric functions on Lipschitz domains vanishing on a piece of the boundary, are as follows (see \cite{FS1}, \cite{FS2}):

\textit{Interior Harnack Inequality}: There exists a positive constant $c=c(n,\lambda,\Lambda)$ such that for any $r\in(0,1)$
\[
u(\underline{P_r}) \leq cu(\overline{P_r}).
\]

\textit{Carleson Estimate}: There exists a $c=c(n,\lambda,\Lambda,L,L_0)$ and $\beta=\beta(n,\lambda,\Lambda,\\L,L_0)$, $0<\beta\leq 1$ such that for every $X\in \Omega_{r/2}$
\[
u(X) \leq c\left(\frac{\delta_X}{r}\right)^\beta u(\overline{P_r}).
\]

\textit{Boundary Harnack Principle}: There exists $c=c(n,\lambda, \Lambda, L, L_0)$ and $\beta=\beta(n,\lambda, \Lambda, L, L_0)$, $0<\beta \leq 1$, such that for every $(x,t) \in \Omega_{2r}$ and $u$ and $v $ are two solutions
\[
\frac{u(x,t)}{v(x,t)} \geq c \frac{u(\underline{P}_r)}{v(\bar{P}_r)}.
\]

\textit{Backward Harnack Inequality}: Let $m=u(\underline{P_{3/2}})$ and $M=\sup_{\Omega_2}u$. Then there exists a positive constant $c=c(n,\lambda,\Lambda,L,L_0,M/m)$ such that if $r\leq 1/2$
\[
u(\overline{P_r}) \leq cu(\underline{P_r}).
\]

 We will use $c$ to denote constants which depend on some or all of $n$, $\lambda$, $\Lambda$, $L$, $L_0$, $M/m$. We will write $\Gamma(\theta,\eta)$ to denote a cone of directions with axis $\eta$ and opening $\theta$. 
\begin{defn}(Monotonicity)
A function $u\geq 0$ is  $\varepsilon_0$-monotone in a domain $\Omega$ in a cone of directions $\Gamma(\theta, e_n)$ if there exists a $\bar{\beta} >0$ such that for any $\ep\geq\ep_0$, $\tau \in \Gamma(\theta,e_n)$ 
 \[
 u(p) -u(p-\varepsilon \tau) \geq c\varepsilon^{\bar{\beta}} u(p)
 \]
 provided  both $p$ and $p-\ep\tau$ belong to $\Omega$.
A function $u\geq 0$ is fully monotone in the direction $\tau$ if for any $\ep>0$
 \[
 u(p) -u(p-\varepsilon \tau) \geq 0.
 \]
 A function $u$ of arbitrary sign is $\ep_0$-monotone if both $u^+$ and $-u^-$ are $\ep_0$-monotone. 
\end{defn}

The following lemma, derived from Lemma 2.3 in \cite{FS1}, provides a starting point for our problem.
 \begin{lm}
 Let $u$ be a viscosity solution to~\eqref{FBP_Statement}.
 Then there exists a cone of directions $\Gamma(\theta,e_n)$ in which $u$ is fully monotone in space and $\varepsilon_0$-monotone in time.
 \end{lm}
 \textbf{Remark:} By a rescaling argument we may assume that $u$ is fully monotone in both space and time in the cone $\Gamma(\theta,e_n)$ outside an $\varepsilon_0$ neighborhood of \fb(u) (see remarks at beginning of section 4 in \cite{FS1}). \\

 We  list some technical constants used in the iteration: Let $\varepsilon_0$ be a small fixed number which is the $\varepsilon_0$-monotonicity. Let $\beta, \delta, \gamma$ be positive numbers such that
 \[
 0<\gamma =\frac{1-\delta}{2}, \quad 0<\beta < \min\left\lbrace \frac{1-\delta}{2}, \frac{\alpha+\delta-1}{2} \right\rbrace.
 \] Here $\alpha$ is the Holder exponent of the coefficients and $\delta$ will be the defect angle of the cone of monotonicity ($\delta =\pi/2-\theta$).

 Later we will show at each step of an iteration $\varepsilon$-monotonicity with a suitable $\varepsilon <\varepsilon_0$ and $\bar{\beta}=1-\gamma+\beta$. But first we quote two results from \cite{FS1} that find application to this problem. Although the problem in \cite{FS1} is the Stefan problem, these results only depend on the Lipschitz nature of the domain, the fact that $u$ vanishes along the graph of $f$ and that $u$ is a $\mathcal{L}$-caloric function. Therefore they are valid in our case.

 \begin{lm}\label{ln:FullMonoAway} (Lemma 2.4 in [FS1])\\
 \indent Let $\alpha \leq 1$ be the Holder exponent of the $a_{ij}(x,t)$ and let $\beta, \delta, \gamma$ be chosen as above.

 Suppose $u\geq 0$ is monotone in the $e_n$ direction and
 \[
 u(p) -u(p-\varepsilon\tau) \geq c\varepsilon^{1-\gamma+\beta} u(p)\geq c\varepsilon^{\frac{1+\delta}{2}+\beta}u(p)
 \]
 for $d_p < \eta/4$ ($d_p$ is the distance from $p =(x,t)$ to the free boundary at time level $t$) where $\tau =\beta_1 e_n +\beta_2 e_t$ with $\beta_1 >0$, $|\beta_2| \neq 0$ and $|\tau| =1$. Then if $M =M(n,L)$ is large enough and $\varepsilon$ is small enough outside a $(M\varepsilon)^\gamma$ neighborhood of $F(u)$ we have
 \[
 D_{\tau_\varepsilon} u \geq 0
 \] where $\tau_\varepsilon =\tau + c(M\varepsilon)^{(\alpha+ \delta -1)/2} e_n$ for some $c=c(n,L,L_0,\beta_1,\beta_2)$

 \end{lm}

 \textbf{Remark:} The thrust of the lemma is that a spacially monotone solution $u$ which is $\varepsilon$ monotone in a space-time direction $\tau$ as above will be fully monotone in a slightly different space-time direction $\tau_\varepsilon$ if far enough from the free boundary. Note that if the original $\tau$ came from some cone of $\varepsilon$-monotonicity then $u$ will be fully monotone in a smaller cone of directions away from the free boundary. The term $c(M\varepsilon)^{(\alpha+ \delta -1)/2} e_n$  is describes the amount of the cone that must be given up to obtain full monotonicity. In what follows, the main idea will be an iteration that gives up a certain amount of the cone at each step in order to reduce the $\varepsilon$-monotonicity at each step, ultimately proving a cone of full monotonicity exists for $u$. We explicitly observe  that, by rescaling, we may assume that Lemma holds \ref{ln:FullMonoAway} in $Q_1$.\\

The other ingredient our work will require is the following family of functions. We begin with the domains involved. Let
 \[
 \mathcal{N}_{b\varepsilon} =\{p=(x',x_n,t): d(p,\fb(u)) <b\varepsilon\} \quad b>2L
 \]
\[
\mathcal{C}_{b,R,T} = \mathcal{N}_{b\varepsilon} \cap \{|x'|<R\}\cap \{|t|<T\}.
\]
Here $d$ is the ordinary distance. We denote by $\Omega_{\varepsilon,R,T}$ a smooth domain with
\[
\mathcal{C}_{b/2,R,T} \subset \Omega_{\varepsilon,R,T} \subset \mathcal{C}_{b,R,T}
\]

\begin{lm}\label{lm:FamilySubSlns}(Lemma 3.4 in \cite{FS1}). Let C, $c_0, b_0$, $\omega_0$ be positive numbers. Choose positive numbers $\beta, \gamma, \delta$ as above and $\tilde{\alpha}$ such that $0<\tilde{\alpha}< 1-\beta$. If $C>1$ and $\omega_0$ is small enough, there exists a family of functions $\phi_\eta$ such that $\phi_\eta \in C^2(\bar{\Omega}_{\varepsilon,R,T})$, $0\leq \eta \leq 1$ and
\begin{enumerate}
\item[(a)] $0\leq 1-\omega_0 \leq\phi_\eta \leq 1+\eta-\omega_0$
\item[(b)] $\phi_\eta(\mathcal{L}\phi_\eta -D_t \phi_\eta -|\nabla \phi_\eta |) \geq C(|\nabla \phi_\eta|^2 +\omega_0^2)$
\item[(c)] $|D_t \phi_\eta | \leq c\varepsilon^{-\tilde{\alpha}}, |\nabla \phi_\eta | \leq c\varepsilon^{\beta-1}$
\item[(d)] $D_t \phi_\eta \geq 0$
\item[(e)] $\phi_\eta \leq 1$ in
\[
\bar{\Omega}_{\varepsilon,R,T} \cap \left(\{-T<t<-T+\varepsilon^{\tilde{\alpha}} \} \cup \{ x: R-\frac{\varepsilon^{\tilde{\alpha}/4}}{2} < |x'| <R\}\right)
\]
\item[(f)] $\phi_\eta \geq 1-\omega_0+\eta(1-c\varepsilon^\beta)$ in
\[
\bar{\Omega}_{\varepsilon,R,T} \cap \left(\{t>-T+2\varepsilon^{\tilde{\alpha}} \} \cup \{ x:|x'| < R-\frac{\varepsilon^{\tilde{\alpha}/4}}{2} \}\right).
\]

\end{enumerate}
\end{lm}

\textbf{Remarks:} (1) We will apply this Lemma with $\omega_0$ being the oscillation of the coefficient matrix $A_{ij}$. Since we are assuming Holder continuous coefficients, we can assume that $\omega_0$ is small, as a rescaling depresses the oscillation.\\
(2) In \cite{FS1}  the $\phi_\eta$ are defined as
\[
\phi_\eta(x,t)=1+\omega_0(|x'|^2-1) +\eta\left(\frac{F(x,t)-1}{2^{\frac{1}{2C-1}-1}}\right)
\]
One clearly sees that the $\phi_\eta$ vary continuously in $\eta$. In turn, this means that $v_\eta$'s (defined below) and their free boundaries vary continuously in $\eta$ as well.\\
(3) We will be interested in constructing a family of sup-convolutions
 \[
 v_\eta(p) = \sup_{B_{\sigma\phi_\eta(p)}(p)} u_1
 \] where $B_{\sigma\phi_\eta(p))}(p)$ is the ball of radius $\sigma\phi_\eta(p)$ centered at $p$, with $\sigma$ to be chosen later and $u_1(q) =u(q-\lambda\varepsilon\tau)$. With this in mind, we clarify the role that the conclusions of the lemma have on this family.
 Condition $(b)$ is the most important. In fact, it is proved in \cite{FS1} that if a function satisfies condition $(b)$, then the function $v_\eta$ will be a $\mathcal{L}$-subsolution in its positive and negative set. Conditions $(e)$ and $(f)$ involve how the family should behave near the boundary of their domain. Near the boundary no gain in monotonicity can be expected, and hence $\phi_\eta \leq 1$ on this region. In the interior, where gain is expected, we have $\phi_\eta \geq 1-\omega_0 +\eta(1-c\varepsilon^\beta)$, and thus there is a definite increase in the radius of the balls over which the supremum is taken. \\

\section{Asymptotic Developments}
 This section describes the behavior of the solution $u$ near its free boundary. The results of this section also find application to the behavior of the sup-convolutions near their zero sets. We explicitly remark that such results are valid for any $\mathcal{L}$-caloric functions vanishing on a distinguished piece of the boundary of a Lipschitz domain. The following result from \cite{FS1} (Lemma 3.5; see also 13.19 in \cite{CS}) provides the first result along these lines.
\begin{lm}\label{lm:AD_for_u}
Let u be $\mathcal{L}$-caloric in the open set $D$, vanishing on $F =\partial D \cap C_1$. Supposed that $(0,0) \in F$ and there is an $(n+1)$-dimensional ball $B$ such that $\bar{B}\cap F =\{(0,0)\}$. Assume that the tangent plane to $B$ is given by
\[
\beta^++\alpha^+\langle x, \nu \rangle =0
\]
for some spatial unit vector $\nu$ and some real numbers $\alpha^+,\beta^+, \alpha^+>0$,  ($-\beta^+/\alpha^+$ finite). Then, either $u$ grows more than any linear function or:
\begin{enumerate}
\item[(a)] ($B \subset D$) Then, near $(0,0)$, for $t\leq 0$ \[
u(x,t) \geq [\beta^+ t + \alpha^+ \langle x,\nu \rangle]^+ + o(d(x,t)).
\]
\item[(b)] ($B \subset D^C$) Then, near $(0,0)$, for $t\leq 0$ \[
u(x,t) \leq [\beta^+ t + \alpha^+ \langle x,\nu \rangle]^+ + o(d(x,t)).
\] Furthermore, equality holds in both case along paraboloids of the form $t=-\gamma \langle x,\nu \rangle$ $\gamma>0$.
\end{enumerate}

\end{lm}

 \textbf{Remark:}  By applying Lemma \ref{lm:AD_for_u} to both $u^+$ and $(-u)^+$, where $u$ is our two-phase viscosity solution, we can infer that if the origin is a right regular point for $u^+$, and hence a left regular point for $(-u)^+$, then
 \[
 u(x,t) \geq (\beta^+t+\alpha^+(x,\nu))^+ - (\beta^-t+\alpha^-(x,\nu))^- +o(d(x,t)).
 \]
 Likewise, if the origin is regular from the left for $u$ we have the same statement with the inequality reversed.

 \begin{lm} \label{lm:Asym_Bndry_Con}
 Let $u$ be a viscosity solution to our free boundary problem in $Q_1$ with FB($u$) Lipschitz, $(0,0)\in $ FB($u$) and suppose that in a neighborhood of the origin with $t\leq 0$ we have for $\alpha^+ >0$, $\alpha^-\geq 0$
 \[
 u(x,t) \geq (\beta^+t + \alpha^+\langle x,\nu \rangle)^+ - (\beta^-t+ \alpha^-\langle x,\nu \rangle)^- +o(d(x,t)).
 \]
 Then $G(\alpha^+,\alpha^-) \leq 1$.

 Likewise, if for $\alpha^+ \geq 0$, $\alpha^-> 0$
 \[
 u(x,t) \leq (\beta^+t + \alpha^+\langle x,\nu \rangle)^+ - (\beta^-t+ \alpha^-\langle x,\nu \rangle)^- +o(d(x,t)).
 \] Then $G(\alpha^+,\alpha^-) \geq 1$.

 \end{lm}
\begin{proof}We prove the first statement.
Suppose that the conclusion of the lemma is false. Then there exists an $\eta >0$ such that $G(\alpha^+,\alpha^-) \geq 1+\eta >1$. We can assume that $\nu =e_n$ for simplicity.

Let $R$ be a small parabolic neighborhood of the origin. Set
\[
\psi(x,t) = \bar{\alpha}^+x_n +{\beta}^+t -ct^2 +\frac{2\Lambda}{\lambda} x_n^2 -\frac{|x'|^2}{2(n-1)}
\] where $\bar{\alpha}^+ = \alpha-\ep$ with $\ep$ to be determined later.

Choose $c>0$ large so that the level surface $\{\psi=0\}$ is strictly convex and $\{\psi>0\}\cap R \subset \Omega^+(u)$

Let $\mathcal{L}_r$ =$\frac{1}{r}\sum_{i,j} a_{ij}(rx,rt)$. Observe that
\[
\begin{aligned}
\mathcal{L}_r\psi -\psi_t &= \frac{1}{r}\sum_{i,j} a_{ij}D_{ij}\psi -[{\beta}^+-2ct] \\
&= \frac{1}{r}\left[-\frac{1}{n-1}(a_{11} +a_{22} +...+a_{n-1,n-1})+\frac{2\Lambda}{\lambda}a_{nn}\right]-[{\beta}^+-2ct]\\
&\geq  \frac{\Lambda}{r} -[{\beta}^+-2ct]\\
&>0
\end{aligned}
\] provided $r$ is small enough.\\

\textbf{Claim:} \textit{If $\varepsilon$ and $R$ are small enough, then the function
\[
\phi =\psi^+ -\frac{\bar{\alpha}^-}{\bar{\alpha}^+}\psi^-,
\] with $\bar{\alpha}^- =\alpha^- +\varepsilon$, is a classical strict $\mathcal{L}_r$-subsolution in R}.

To prove the claim, note that by the observation above, if $r$ is small enough, $\phi$ will satisfy the subsolution condition away from its free boundary. So we only need to establish the free boundary condition $G$.

If $\ep$ is sufficiently small, the continuity of $G$ and our assumption that $G(\alpha^+,\alpha^-)\geq 1+\eta$ implies that
\[
G(\phi^+_n,\phi^-_n) \geq G(\alpha^+,\alpha^-)-\eta/2 \geq 1+\eta/2.
\]
Using the continuity of $\phi$ and $G$ we may assume that the subsolution condition holds throughout $R$, assuming this region is small enough. That is, along the free boundary of $\phi$ in $R$ we have
\[
G(\phi^+_\nu,\phi^-_\nu)) \geq 1+\eta/4 >1.
\] 
  This means that $\phi$ is a classical strict subsolution in $R$. Note that this choice of $R$ does not depend on $r$.

 Now define
 \[
 u_r(x,t) =\frac{u(rx,rt)}{r}.
 \]
Then $u_r$ is a viscosity $\mathcal{L}_r$-solution. Furthermore,
\[
u_r(x,t) \geq (\beta^+t + \alpha^+\langle x,\nu \rangle)^+ - (\beta^-t+ \alpha^-\langle x,\nu \rangle)^- +o_r(1)
\]
where $o_r(1)$ denotes the decay of the error term as a function of $r$. If we choose $r$ to be very small, we will then have
\[
u_r(x,t) \geq (\beta^+t + \alpha^+\langle x,\nu \rangle)^+ - (\beta^-t+ \alpha^-\langle x,\nu \rangle)^- +o_r(1) >\phi(x,t)
\] on the boundary of our neighborhood $R$.  This strict inequality ought to propagate to the interior since $u_r$ is a viscosity solution and $\phi$ a subsolution, but we have $u_r(0,0)=0=\phi(0,0)$ a contradiction.
\end{proof}

 \begin{lm}\label{lm:Prop_of_SupCon}
Let $u$ be a viscosity solution to our free boundary problem with a Lipschitz \fb. Define $u_1= u(p-\lambda\varepsilon\tau)$
\[
v_\eta(x,t) = \sup_{B_{\sigma\phi_\eta}} u_1
\] where $\phi_\eta$ is as Lemma~\ref{lm:FamilySubSlns}. Assume that this sup is attained uniformly away from the top and bottom of the ball. Then the following hold:
\begin{enumerate}
\item $v_\eta$ is a subsolution to our equation on $\Omega^+(v_\eta)$ and $\Omega^-(v_\eta)$.
\item All points on $\fb(v_\eta)$ are regular from the right.
\item $\fb(v_\eta)$ is uniformly Lipschitz in space and time
\item If $(x_0,t_0) \in \fb(v_\eta)$ and $(y_0,s_0)\in \fb(u_1)$ with
\[
(y_0,s_0) \in \partial B_{\sigma\phi_\eta(x_0,t_0)}(x_0,t_0)
\]
then $(x_0,t_0)$ is a regular point from the right. Moreover if near $(y_0,s_0)$ along the paraboloid $s=s_0-\gamma\langle y-y_0,\nu \rangle^2$, ($\gamma>0$), $u_1$ has the asymptotic expansion
\[
u_1(y,s) = \alpha^+\langle y-y_0,\nu\rangle^+ -\alpha_-\langle y-y_0,\nu\rangle^+ +o(|y-y_0|)
\] $\nu = \frac{y_0-x_0}{|y_0-x_0|}$, then near $(x_0,t_0)$ along the paraboloid $t=t_0-\gamma\langle x-x_0,\nu \rangle^2$ we have
\[
\begin{aligned}
v_\eta \geq &\alpha_+(x-x_0,\nu + \frac{\sigma \phi_\eta(x_0,t_0)}{|y_0-x_0|}\nabla(\sigma\phi_\eta))^+\\
&-\alpha_-(x-x_0,\nu + \frac{\sigma \phi_\eta(x_0,t_0)}{|y_0-x_0|}\nabla(\sigma\phi_\eta))^- +o(|x-x_0|).
\end{aligned}
\]
\end{enumerate}
\end{lm}
 \begin{proof}
 (1) is proved in \cite{FS1}, Lemma 3.1.\\
 (2) \& (3) are standard facts. See Lemma 9.13 in \cite{CS}.\\
 (3) For $x$ near $x_0$ take $t$ to be on the corresponding paraboloid. Set $y=x+\nu\bar{\phi}(x)$ (we suppress the $\eta$ subscript for convenience; the result is to hold for any choice of $\eta$) where
 \[
 \bar{\phi}(x) = \sqrt{\phi^2(x,t)-(s_0-t_0)^2} \leq \phi(x,t).
 \]
 Note that $t$ depends on $x$ here since $(x,t)$ is to lie on the paraboloid. Given this $y$, let $s$ be the corresponding time value so that $(y,s)$ lies on the paraboloid for $u_1$. Then $v_\eta(x,t) \geq u_1(y,s)$ since $(y,s)$ is in the ball over which we are taking the sup.

 Since
\[
\nabla \bar{\phi} \mid_{x_0} =\frac{\phi(x_0,t_0)}{\bar{\phi}(x_0)}\nabla \phi (x_0,t_0).
\]
and $\bar{\phi}(x_0)=|x_0-y_0|$, we can write $\bar{\phi}$ as
 \[
 \bar{\phi}(x) = \bar{\phi}(x_0) +\langle x-x_0,\frac{\phi(x_0,t_0)}{|y_0-x_0|}\nabla \phi(x_0,t_0) \rangle +o(|x-x_0|).
 \]

 Now we compute
 \[
 \begin{aligned}
 \langle y-y_0,\nu \rangle = &\langle x-x_0+(\bar{\phi}(x)-\bar{\phi}(x_0))\nu,\nu \rangle +o(|x-x_0|) \\
  = & \langle x-x_0+(\langle x-x_0,\frac{\phi(x_0,t_0)}{|y_0-x_0|}\nabla \phi(x_0,t_0) \rangle)\nu,\nu \rangle +o(|x-x_0|)\\
  = & \langle x-x_0, \nu \rangle + (\langle x-x_0,\frac{\phi(x_0,t_0)}{|y_0-x_0|}\nabla \phi(x_0,t_0) \rangle) \langle \nu, \nu \rangle +o(|x-x_0|)\\
  = & \langle x-x_0,\nu +\frac{\phi(x_0,t_0)}{|y_0-x_0|}\nabla \phi(x_0,t_0) \rangle +o(|x-x_0|).
 \end{aligned}
 \]
 Now substitute this result into the asymptotic behavior of $u_1$ and use the fact that $v_\eta(x,t_0) \geq u_1(y,s_0)$ to reach the desired conclusion.

 \end{proof}
 \textbf{Remark} The condition on the sup being attained uniformly away from the top and bottom of the ball is implicitly used to ensure that $|x_0-y_0|$ is bounded away from zero in the above computations. The condition is not restrictive; for a solution with Lipschitz free boundary satisfies it near the FB, and by rescaling, one may assume that it occurs throughout $C_1$. \\

 \section{Monotonicity up to the Free Boundary}
Thus far in this work we have used $\varepsilon_0$-monotone in a direction $\tau$ to mean that for any $\ep \geq \ep_0$
 \[
 u(p) -u(p-\varepsilon \tau) \geq c\varepsilon^{\bar{\beta}} u(p).
 \] 
 At this point however, it is more convenient to work with a formulation of $\ep_0$-monotonicity more compatible with our sup-convolutions $v_\eta$. In this formulation, we say that $u$ is $\ep_0$-monotone in a cone of directions $\Gamma(\theta,e_n)$ if 
 \[
 \sup_{B_{\varepsilon\sin\delta}(p)} u(q-\varepsilon\tau) \leq u(p)
 \]
  for any $\tau \in \Gamma(\theta -\delta,e_n)$ and $\ep\geq \ep_0$. Here $\delta = \pi/2 -\theta$ is the defect angle of the cone; if it is zero the cone would be a half-space. Throughout all our work we assume that $\delta \ll \theta$. 
  The two formulations of $\ep$-monotonicity are essentially equivalent, with perhaps a slight difference in $\ep$, so our previous results translated into this formulation become: For $\tau \in \Gamma(\theta-\delta,e_n)$ and $\varepsilon\geq \varepsilon_0$
  \[ 
 \sup_{B_{\varepsilon\sin\delta}(p)} u(q-\varepsilon\tau) \leq (1-c\varepsilon^{\bar{\beta}}) u(p).
 \]

  Our goal in what follows is to decrease $\varepsilon$ by a factor of $\lambda<1$. Let $\sigma =\varepsilon(\sin\delta -(1-\lambda))$. Now if $\lambda$ is close enough to one 
  \[
  B_\sigma (p -\lambda\varepsilon\tau) \subset B_{\varepsilon\sin\delta}(p-\varepsilon\tau).
  \]
  This means
 \[
 \sup_{B_\sigma(p)} u(q-\lambda\varepsilon\tau) \leq (1-c\varepsilon^{\bar{\beta}}) u(p).
 \]
 This has decreased $\varepsilon$ but at the expense of reducing the ball's radius. 
 
 However, away from the free boundary at a distance of $(M\varepsilon)^\gamma$, by Lemma~\ref{ln:FullMonoAway} our solution $u$ is in fact fully monotone. So away from the free boundary there is no need to give up any of the radius and we have by the definition of monotone
 \[
 \sup_{B_{\varepsilon\lambda\sin\delta}(p)} u(q-\lambda\varepsilon\tau) \leq u(p).
 \] 
 We will use the variable family of radii $v_\eta$ to bridge these two conclusions. Recall $\tilde{\alpha}$ occurs in the definition of the $\phi_\eta$ and $\bar{\beta}=1-\gamma+\beta$ is defined in Section 2.\\

\begin{thrm} Let $u$ be a solution to our FBP in $C_{R,T} =B'_R \times (-T,T)$ such that
\begin{enumerate}
\item[(i)] u is monotonically increasing along the directions of a spacial cone $\Gamma^x(\theta,e_n)$ with $\pi/2 -\theta \ll 1$.
\item[(ii)] u is $\varepsilon$-monotone in a space time cone of directions $\Gamma(\theta_*, e_n)$
\item[(iii)] u is monotone along the directions $\tau \in \Gamma(\theta_*,e_n)$ outside an $\varepsilon$-neighborhood of $\fb(u)$.
\item[(iv)] The non-degeneracy condition holds for $u$ at regular points from the right.
\end{enumerate}
Then there exists an $\varepsilon_0$ and a $\lambda$, $0<\lambda <1$ such that if $\varepsilon\leq \varepsilon_0$ we have in $C_{R-c\varepsilon^{\bar{\alpha}},T-c\varepsilon^{\bar{\alpha}}} $ $u$ is $\lambda\varepsilon$-monotone in $\Gamma(\theta_*-\bar{c}\varepsilon^\beta,e_n)$.
\end{thrm}
  \begin{proof} For the purposes of this proof, let $u_1(p) =u(p-\lambda\varepsilon \tau)$.

Since we have ($\textit{iii}$), we only need to show the improvement in an $\varepsilon$-neighborhood of the free boundary. Precisely, our goal is to show the improvement in $\Omega_{\varepsilon,R,T} \cap C_{R-c\varepsilon^{\bar{\alpha}}, T-c\varepsilon^{\bar{\alpha}}}$.\\

  Define
 \[
 v_\eta(p) = \sup_{B_{\sigma\phi_\eta}(p)} u(q-\lambda\varepsilon\tau).
 \]
 Choose $\bar{\eta}$ such that
 \[
 \sigma(1-\omega_0 +\bar{\eta}) = \varepsilon(\lambda\sin \delta -c\varepsilon^\beta)
 \]
 where $\sigma =\varepsilon(\sin\delta -(1-\lambda))$. Assume our defect angle $\delta$ is small and take $(1-\lambda)=\frac{1}{2}\sin\delta$. Then, since $0 <\varepsilon \ll \delta \ll 1$, we have that $1/3 <\bar{\eta}<1$ (the lower bound is the one of consequence here).  We proceed to perturb this function as follows:
 \[
 \bar{v}_\eta = v_\eta +c\varepsilon^{1+\beta -\gamma}w_\eta.
 \]
\begin{comment} The reason we need this is because we know that $v_\eta$ has the following asymptotic behaviour: (see lemma 3 above)
 \[
\begin{aligned}
v_\eta \geq &\alpha_+(x-x_0,\nu + \frac{\sigma \phi_\eta(x_0,t_0)}{|y_0-x_0|}\nabla(\sigma\phi_\eta))^+\\
&-\alpha_-(x-x_0,\nu + \frac{\sigma \phi_\eta(x_0,t_0)}{|y_0-x_0|}\nabla(\sigma\phi_\eta))^- +o(|x-x_0|)
\end{aligned}
\]
We know that $G(\alpha_+,\alpha_-) \geq 1$ from the behaviour of $u$. But $v_\eta$ has a slightly different behaviour with an $\bar{\alpha}_+ = \alpha_+|\nu^*|$ where
\[
 \nu^*=(y_0-x_0+\sigma^2 \phi_\eta |\nabla \phi_\eta|)/|y_0-x_0|
 \]
 so we could only conclude that $\bar{\alpha}_+ \geq \alpha_+(1-c\sigma^2|\nabla \phi_\eta|)$ which is not good enough to make $v_\eta$ a strict subsolution in its own right (we can't guarantee the needed $G(\bar{\alpha}_+,\alpha_-) >1$ condition). $G$ is strictly increasing in its first argument so we need a way to boost the $\bar{\alpha}_+$. We do this by adding in the $w_\eta$. \\\end{comment}

 We define $w_\eta$ as follows:
\[
 \left\lbrace
 \begin{aligned}
      \mathcal{L}w_\eta-(w_\eta)_t &=0 \quad \text{in} \quad \Omega^+(v_\eta)\cap \Omega_{\varepsilon,R,T}\\
      w_\eta &=0 \quad \text{on} \quad \fb(v_\eta)\\
      w_\eta &=u \quad \text{on rest of} \quad \partial_p[\Omega_{\varepsilon,R,T} \cap \Omega^+(v_\eta)].
 \end{aligned}
  \right.
\
\] Extend $w_\eta$ by zero on the rest of $\Omega_{\varepsilon,R,T}$.
 Note that $w_\eta\leq u$ in $\Omega^+(v_\eta) \cap \Omega_{\varepsilon,R,T}$ by the maximum principle.\\

 We will show that for every $\eta \in [0,\bar{\eta}]$ we have
 \[
 \bar{v}_\eta \leq u
 \] in $\Omega_{\varepsilon,R,T} \cap C_{R-c\varepsilon^{\bar{\alpha}}, T-c\varepsilon^{\bar{\alpha}}}$. This will be accomplished by showing that the set of $\eta$'s for which $\bar{v}_\eta \leq u$ is non-empty and both open and closed. The set being closed follows from the fact that the $\bar{v_\eta}$ vary continuously in $\eta$ (see the remark after Lemma~\ref{lm:FamilySubSlns}), so we only need to show that it is  non-empty and open.

 \paragraph*{Non-Empty:}

 We show that $\bar{v}_0 \leq u$. Now from the properties of $\phi_\eta$, $\phi_0 \equiv 1-\omega_0$, so that
 \[
 v_0 =\sup_{B_{\sigma(1-\omega_0)}} u(q-\lambda\varepsilon\tau).
 \]
 Now clearly $\sigma \geq (1-\omega_0)\sigma$. Using the fact that
 \[
 \sup_{B_\sigma(p)} u(q-\lambda\varepsilon\tau) \leq (1-c\varepsilon^{\bar{\beta}}) u(p),
 \] we see that $v_0(p) \leq (1-c\varepsilon^{1+\beta-\gamma})u(p)$ (recall that $1+\beta-\gamma =\bar{\beta}$). Now by choosing a suitable new constant $c$, we can arrange $\bar{v}_0  =v_0+c\varepsilon^{1+\beta-\gamma}w_0 \leq u$.

 \paragraph*{Open:}

 We claim that it will be enough to show that
 \begin{equation}\label{eq.inclusion}
 \Omega_{\varepsilon,R,T}\cap \{\bar{v}_\eta >0\}\subset\subset\Omega_{\varepsilon,R,T}\cap \{u >0\}
  \end{equation}
  for every $\eta \in [0,\bar{\eta})$. We argue this claim as follows:

 First, note that $u$ is fully monotone, and thus $\lambda\varepsilon$-monotone, outside of $\Omega_{\varepsilon,R,T}$ by hypothesis $(iii)$. By our choices of $\sigma$ and $\bar{\eta}$ at the beginning of this proof, we have that $\sigma\phi_\eta \leq \lambda\varepsilon \sin \delta$ for any $0\leq \eta \leq 1$. This means that outside of $\Omega_{\varepsilon,R,T}$
 \[
 v_\eta (p) \leq \sup_{B_{\lambda\varepsilon\sin \delta}} u(p-\lambda\varepsilon\tau) \leq (1-c(\lambda\varepsilon)^{\bar{\beta}})u(p)
 \]  for any $\eta \in [0,1]$.
 Using this, the assumption that $\lambda$ is close to one, and adjusting the constant $c$ that appears in the definition of $\bar{v}_\eta$, we have that $\bar{v}_\eta \leq u$ along the boundary of $\Omega_{\varepsilon,R,T}$.

  Now assume $\Omega_{\varepsilon,R,T}\cap \{\bar{v}_\eta >0\} \subset\subset \Omega_{\varepsilon,R,T}\cap \{u >0\}$. We want to show that $\bar{v}_\eta \leq u$. By the preceding argument, we have that $u \geq \bar{v}_\eta$ on $\partial_p[ \{ \bar{v}_\eta>0\}\cap\Omega_{\varepsilon,R,T}]$. Inside this domain both functions are classical solutions so the maximum principle implies that $u \geq \bar{v}_\eta$ inside as well.  An entirely similar argument, though somewhat simpler since $w_\eta$ does not appear, yields the same conclusions for $\{\bar{v}_\eta<0\}$. Hence $\bar{v}_\eta \leq u$ everywhere. \\

 We now proceed to prove \eqref{eq.inclusion} by contradiction. Assume there exists a point $(x_0,t_0) \in \fb(u)\cap \fb(v_\eta)$ for some $\eta$ which we now regard as fixed. All points on $\fb(v_\eta)$ are regular from the right and, since the $\fb(u)$ touches $\fb(v_\eta)$ at $(x_0,t_0)$, we have that this point is regular from the right for $u$.

 By the definition of $v_\eta$, we have that there is a corresponding point $(y_0,s_0) \in \partial B_{\sigma\phi_\eta(x_0,t_0)}(x_0,t_0) \cap \fb(u_1)$. This point is then a regular point from the left for $u_1$. By Lemma~\ref{lm:AD_for_u} and Lemma~\ref{lm:Asym_Bndry_Con} we have along parabolas of the form $s=s_0 -\gamma\langle y-y_0,\nu \rangle^2$
\[
 u_1(y,s) = \alpha_+\langle y-y_0,\nu\rangle^+ - \alpha_-\langle x,\nu\rangle^- +o(d(x,t))
 \] where $\nu =\frac{y_0-x_0}{|y_0-x_0|}$ with $G(\alpha_+,\alpha_-) \geq 1$.

 Next, by Lemma~\ref{lm:Prop_of_SupCon}, we have that near $(x_0,t_0)$ along the the parabola $t=t_0 -\gamma \langle x-x_0,\nu \rangle^2$ we have
% Therefore, by Lemma~\ref{lm:AD_for_u} we have at the $s_0$ level
%\[
% u_1(y,s_0) = \alpha_+\langle y-y_0,\nu\rangle^+ - \alpha_-\langle y-y_0,\nu\rangle^- +o(|y-y_0|).
% \]
%Here $\nu =\frac{y_0-x_0}{|y_0-x_0|}$. Additionally, by applying Lemmas~\ref{lm:AD_for_u} and~\ref{lm:Asym_Bndry_Con}, we obtain  $G(\alpha_+,\alpha_-) \geq 1$.
%
% Next, by Lemma~\ref{lm:Prop_of_SupCon}, near $(x_0,t_0)$ we have
 \[
 \begin{aligned}
v_\eta \geq &\alpha_+(x-x_0,\nu + \frac{\sigma \phi_\eta(x_0,t_0)}{|y_0-x_0|}\nabla(\sigma\phi_\eta))^+\\
&-\alpha_-(x-x_0,\nu + \frac{\sigma \phi_\eta(x_0,t_0)}{|y_0-x_0|}\nabla(\sigma\phi_\eta))^- +o(|x-x_0|).
\end{aligned}
 \]
  In other words, 
 \[
  v_\eta \geq \alpha_+|\nu^*|(x-x_0,\nu^*)^+ -\alpha_-|\nu^*|(x-x_0,\nu^*)^+ +o(|x-x_0|),
 \]
 where $\nu^* =\nu + \frac{\sigma \phi_\eta(x_0,t_0)}{|y_0-x_0|}\nabla(\sigma\phi_\eta)$.\\

We now turn to the behavior of $w_\eta$ and invoke the non-degeneracy condition we have on $u^+$. This implies a non-degeneracy condition on $u_1$, and in turn on $v_\eta$ via Lemma~\ref{lm:Prop_of_SupCon} . Precisely, by non-degeneracy, the $\alpha_+$ appearing in the asymptotic behavior of $u_1$ is strictly positive bounded away from zero (a consequence of Lemma~\ref{lm:AD_for_u}). Since the same $\alpha_+$ appears in the asymptotic behavior of $v_\eta$, $v_\eta$ also has this non-degeneracy.  By using our Boundary Harnack Principle for $\mathcal{L}$-caloric functions on Lipschitz domains with $w_\eta$ and $v_\eta$ we deduce that $w_\eta$ also possesses a non-degeneracy property, $(w_\eta)_{\nu^*} \geq c >0$ along $\fb(v_\eta) \cap C_{R-c\varepsilon^{\tilde{\alpha}}, T-c\varepsilon^{\tilde{\alpha}}}$.

 Now combining the behavior of $v_\eta$ and $w_\eta$ we have
 \[
 \bar{v}_\eta \geq \bar{\alpha}_+(x-x_0,\nu)^+ -\alpha_-(x-x_0,\nu)^- +o(|x-x_0|)
 \] with
\[
\begin{aligned}
\bar{\alpha}_+ &\geq \alpha_+(1-c\sigma^2|\nabla \phi_\eta |)+c\varepsilon^{1+\beta-\gamma}\\
&\geq \alpha_+(1-c\varepsilon^2\delta\varepsilon^{\beta-1}) +c\varepsilon^{1+\beta-\gamma}\\
&> \alpha_+
\end{aligned}
\] for $\varepsilon$ small enough. We have only used $w$ to perturb the positive part of $v_\eta$, hence the $\alpha_-$ is unchanged.

Since, as noted above, $(x_0,t_0)$ is a regular point from the right for $u$, we have by Lemma~\ref{lm:AD_for_u}
\[
 u(x,t_0) \geq \alpha_+^{(2)}(x-x_0,\nu)^+ -\alpha_-^{(2)}(x-x_0,\nu)^- +o(|x-x_0|)
\]
where $G(\alpha_+^{(2)},\alpha_-^{(2)}) \leq1$.

$G$ is strictly increasing in the first argument so
\[
G(\alpha_+^{(2)},\alpha_-^{(2)}) \leq 1 \leq G(\alpha_+,\alpha_-) < G(\bar{\alpha}_+,\alpha_-)
\]
However, $u-\bar{v}_\eta$ is a nonnegative supersolution in $\Omega^+(\bar{v}_\eta)$, and so we must have $\alpha_-^{(2)} \leq \alpha_-$. Additionally, by the Hopf Principle,  $\alpha_+^{(2)} > \bar{\alpha}_+$. But $G$ is strictly increasing in its first argument, strictly decreasing in its second and thus we arrive at a contradiction. Hence the set is open.

\paragraph*{Conclusions:}
We have proved that $\bar{v}_\eta \leq u$ for any $\eta \in [0,\bar{\eta}]$. In particular, $\bar{v}_{\bar{\eta}}  \leq u$.

Now recall from the construction of the family $\phi_\eta$ that
\[
\phi_\eta \geq 1-\omega_0+\eta(1-c\varepsilon^\beta)
\] in
\[
\bar{\Omega}_{\varepsilon,R,T} \cap \left(\{t>-T=\varepsilon^{\tilde{\alpha}} \} \cup \{ x:|x'| < R-\frac{\varepsilon^{\tilde{\alpha}/4}}{2} \}\right).
\]
Since
\[
 \sigma(1-\omega_0 +\bar{\eta}) = \varepsilon(\lambda\sin \delta -c\varepsilon^\beta)
 \]
  and
  \[
  \sigma = \frac{\varepsilon}{2}\sin \delta
  \]
  we  have $\sigma\phi_{\bar{\eta}} \geq \varepsilon(\lambda\sin \delta -c^*\varepsilon^\beta)$ in this region (here we used the bound$\bar{\eta}> 1/3$).\\

Finally, $\lambda\sin \delta -c^*\varepsilon^\beta > \lambda\sin(\delta-\bar{c}\varepsilon^\beta)$. We conclude that on
\[
\bar{\Omega}_{\varepsilon,R,T} \cap \left(\{t>-T=\varepsilon^{\tilde{\alpha}} \} \cup \{ x:|x'| < R-\frac{\varepsilon^{\tilde{\alpha}/4}}{2} \}\right)
\]
$u$ is $\lambda\varepsilon$-monotone for any direction $\tau \in \Gamma(\theta -\bar{c}\varepsilon_0^\beta,e_n)$.   \\
\end{proof}
\textbf{Remark:} Hypothesis $(i)$ is not used in the proof of the theorem, but it is needed in order to apply Lemma~\ref{ln:FullMonoAway} in the proof of the following corollary. 

\begin{cor} \label{Cor_Mono_Up_2_FB}
Let $u$ be as in Theorem 2 on $C_{1,1}$. Then on a smaller cylinder $C_{5/6,5/6}$ $u$ is fully monotone in a space-time cone of directions. As a result, we have that
\[
|u_t| \leq CD_nu
\] in this region.

\end{cor}
\begin{proof}
From Theorem 2 we can conclude that $u$ is $\lambda\varepsilon$-monotone in the directions $\Gamma(\theta-\bar{c}\varepsilon^\beta,e_n)$. From Lemma~\ref{ln:FullMonoAway}, $(M\lambda\varepsilon_0)^\gamma$ away from the free boundary $u$ is fully monotone in the directions $\tau'= \tau +c(M\lambda\varepsilon_0)^{\alpha+\delta-1}e_n$. So we have that $u$ is fully monotone $(M\lambda\varepsilon)^\gamma$ away from the free boundary in the cone of directions $\Gamma(\theta-\bar{c}\varepsilon^\beta-(M\varepsilon)^{(\alpha+\delta-1)/2},e_n)$.\\
We iterate this. We can achieve full monotonicity of $u$ up to the free boundary on the smaller cylinder provided we can control the cone loss at every step. This amounts to
\[
\theta^t = \theta_*^t - \bar{c}\varepsilon_0^\beta \sum_k \lambda^{\beta k} -(M\varepsilon_0)^{(\alpha+\delta-1)/2} \sum_k c\lambda^{(\alpha+\delta-1)/2}>0
\]
and
\[
\varepsilon_0^{\tilde{\alpha}}\sum \lambda^{k\tilde{\alpha}} <\frac{1}{6} .
\] We observe that this last term controls how far we stay away from the sides of the original cylinder.

 All the sums involved are convergent geometric series, so provided that $\varepsilon_0$ is small enough we can achieve both conditions.
 \end{proof}

 \section{Lipschitz Continuity of Solution}

The existence of a space-time monotonicity cone up to the free boundary established in Corollary~\ref{Cor_Mono_Up_2_FB} allows us to prove Theorem 1.

 \begin{proof}[\textbf{Proof of Theorem 1}]
  We first note that Lemmas 5.1 and 5.2 from \cite{FS1} continue to hold in our case.

 By Corollary~\ref{Cor_Mono_Up_2_FB}
 \[
 |u_t|\leq C|\nabla u|,
 \]
and thus it suffices to show only the boundedness of the spacial gradient across the free boundary.

Let $(x_0,t_0) \in \Omega^+(u)\cap C_{1/2}$ be a point such that $\mathrm{dist}((x_0,t_0), FB(u)) \leq d_0$. Then the $(n+1)$-dimensional ball $B_d(x_0,t_0)$ touches the free boundary at some point $(\bar{x},\bar{t})$, which we will assume to be the origin for simplicity. We then let $h =\mathrm{dist}((x_0,0),(0,0))$. Since the free boundary is Lipschitz, we have that there exists a $c$ (not dependent on $(x_0,t_0)$) such that $cd\leq h \leq d$.

Now set $A=\frac{u(x_0,0)}{h}$. We aim to prove that $A$ is bounded independent of the point $(x_0,t_0)$. Having thus set up the parameters of the proof, we proceed as in \cite{FS1}, and we refer the reader to that source for more details. Seeking contradiction we construct a function  $z$ in a small neighborhood of the origin $Q_s$ which is a subsolution to the equation $\mathcal{L}-\partial_t$. In addition, using that at the origin the spacial normal (by construction) is $e_n$, one can show that $z$ satisfies the properties
\[
z^+_\nu =z^+_n \geq cA \quad \text{and} \quad z^-_\nu =z^-_n \leq\frac{c_1}{As^2},
\]

and

\[
\frac{z^+_t(0,0)}{z^+_n(0,0)}=\beta. 
\]
The last condition essentially describes the fact  that $\beta$ is the speed of the free boundary for $z$. We will show that $A$ too large forces a contradiction.

 We require that $\beta$ satisfy
\[
\frac{1}{3}\lambda\tilde{C}A <\beta <\lambda\tilde{C}A.
\]
 Note that $A$ large implies that $\beta$ is large.

 We recall that $G(u^+_\nu,u^-_\nu)$ is increasing in the first argument and decreasing in the second. So from the estimates on $z$ at the origin above we have
  \[
  G(z^+_n,z^-_n) \geq G(cA,\frac{c_1}{As^2})
  \]

  By increasing $A$ we can make the right hand side as large as we desire. In particular, we can make $G(cA,\frac{c_1}{As^2})>1$. By the continuity of $z$ we obtain that $z$ is a strict subsolution to our problem in a neighborhood of the origin. We can now adjust the speed $\beta$ to be fast enough (possibly increasing $A$ further if need be) so that the free boundary of $z$ stays to the right (i.e. in the positivity set of) the free boundary $u$. This is possible since the speed of $u$ is finite, owing to $u$ having Lipschitz free boundary. By construction $z <u$ on the parabolic boundary of this neighborhood so the same should hold inside by virtue of $u$ being a viscosity solution. But $u(0,0)=z(0,0)$, a contradiction.
 \end{proof}
  
\end{document}